\documentclass[reqno,10pt]{amsart}
\numberwithin{equation}{section}
\usepackage{latexsym}
\usepackage{amsmath}
\usepackage{amssymb}
\usepackage{mathrsfs}
\usepackage{graphicx,colordvi}
\usepackage{upgreek}
\usepackage{ifthen}
\usepackage{color}

\usepackage[T1]{fontenc}
\usepackage[latin1]{inputenc}

\numberwithin{equation}{section}

\newtheorem{defi}{Definition}[section]
\newtheorem{theorem}[defi]{Theorem}

\newtheorem{proposition}[defi]{Proposition}
\newtheorem{remark}[defi]{Remark}

\usepackage{latexsym}
\usepackage{amsmath}
\usepackage{amssymb}
\newcommand{\dH}[1]{\;{\rm d}{\mathcal{H}}^{#1}}

\newcommand{\cE}{{\mathcal E}}

\newcommand{\R}{{\mathbb R}}

\renewcommand{\epsilon}{\varepsilon}

\newcommand{\diver}{\operatorname{div}}
\newcommand{\bb}{{\bb}}

\begin{document}

\title[Analysis of a  Cahn--Hilliard--Canham--Helfrich flow]{Analysis of a  Cahn--Hilliard--Canham--Helfrich system for the evolution of a two-phase membrane}

\author{Harald Garcke}
\address{Fakult\"at f\"ur Mathematik\\
        Universit\"at Regensburg\\
        Regensburg, Germany}
\email{harald.garcke@ur.de}

\author{Mathias Wilke}
\address{Martin-Luther-Universit\"at Halle-Witten\-berg\\
         Institut f\"ur Mathematik \\
         06099 Halle (Saale), Germany}
\email{mathias.wilke@mathematik.uni-halle.de}

%\thanks{This work was supported by a grant from the Simons Foundation (\#426729, Gieri Simonett).}

\subjclass[2010]{}
% 35B40 Asymptotic behavior of solutions
% 35Q30 Navier-Stokes equations
% 35Q35 PDEs in connection with fluid mechanics

%\keywords{Surface fluid, viscous flow,  surface Navier-Stokes equations, ...}

\begin{abstract}
The coupling of the evolution of a surface with evolution equations defined on that surface
is of relevance in many applications and has been  in the focus of interest in the analysis of parabolic PDEs in recent years. In applications the evolution of two-phase vesicles and biomembranes is governed by flows decreasing an energy which involves Canham--Helfrich-type  curvature energies coupled to a Ginzburg--Landau energy. We derive a new Cahn--Hilliard--Canham--Helfrich   system for the evolution of  two-phase membranes. The resulting system is highly non-linear and we use the theory of
quasi-linear parabolic evolution equations in weighted $L_p$-spaces to show the existence of a strong local-in-time solution and hence demonstrate that the derived system is well-posed.
\end{abstract}

\maketitle

\section{Introduction}

The coupling of the evolution of a surface with the evolution of a function defined on that surface has received considerable attention
 in recent years, see, e.g.,  \cite{MayerSimonett98, MayerSimonett99, ElliotStinner10, ElliotStinner10b,BarrettGN17, PozziSt17, BarrettDS17, KovacsLL20,ElliottGK22, AbelsBG23, AbelsBG23b}. In particular, in biological applications the coupling of a  Willmore type flow
with  the evolution of a chemical species on the surface leads to
 a variety of complex membrane morphologies, see \cite{ElliotStinner10, ElliotStinner10b,BarrettGN17}. This is especially  the case when the chemical species undergoes phase separation and the different phases influence parameters of  the Willmore energy. In  this context we refer to work of Elliott and Stinner \cite{ElliotStinner10, ElliotStinner10b} who considered the evolution of two-phase biomembranes
with phase dependent material parameters. They coupled a geometric evolution equation
of Willmore type to an Allen--Cahn equation posed on the evolving surface.
However, in applications typically the chemical species described by the PDE on the surface fulfills a conservation property which is not satisfied by the non-conserving Allen--Cahn equation considered in \cite{ElliotStinner10, ElliotStinner10b}.
We therefore introduce a model that addresses this limitation and
 couple a Willmore-type flow to a Cahn--Hilliard equation which ensures  mass conservation of  the chemical species on the surface.

The energy of the system we  consider is
\begin{equation}
	\cE (\Gamma , c) := \int_\Gamma \left(\frac 12 \alpha (c) (H-\bar{H}(c))^2
	%\dH{d-1}
	 +
	% \int_\Gamma
	\alpha^G(c) G %\dH{d-1}
	+
	%\int_\Gamma
	\left( \frac \gamma 2| \nabla_\Gamma c|^2 +\Psi(c)\right) \right) \dH{d-1}.
\end{equation}
Here, $\Gamma\subset \R^d$ is a closed hypersurface and $c:\Gamma \to \R$ is the concentration of a chemical species on $\Gamma$.
The term $\int_\Gamma \left(\frac 12 \alpha (c) (H-\bar{H}(c))^2  +
% \int_\Gamma
\alpha^G(c) G \right)
\dH{d-1}$
models the elastic bending energy of the hypersurface  and has found wide applications in the mechanics of biomembranes and vesicles.
This part of the energy was introduced by Canham and Helfrich and generalizes the Willmore energy.
We refer to this generalized Willmore energy as the Canham--Helfrich energy and to its associated $L_2$-gradient flow as the Canham--Helfrich flow.
Here, $H$ is the mean curvature (the sum of the principal curvatures) and $G$ is the Gaussian curvature of the hypersurface $\Gamma$. The positive parameter $\alpha$ is the bending rigidity,  $\alpha^G$ is the Gaussian bending rigidity and $\bar{H}$ is the so-called spontaneous curvature.
The parameters $\alpha$, $\alpha^G$ and $\bar{H}$
are material-dependent and, in our model, vary
with the concentration $c$  of a diffusing chemical species.
 We refer to \cite{Baumgart2003} who  studied such models in the context of
 lipid bilayer membranes and observed a rich shape-transition behavior, including bud formation and vesicle fission. They observed that membranes formed from multiple lipid components can
 laterally separate into coexisting liquid phases, or domains, with
 distinct compositions. They claimed that the phenomena observed  may resemble raft formation in cell membranes
 a process highly important in living systems.

 The energetic term $\int_\Gamma
\left( \frac \gamma 2| \nabla_\Gamma c|^2 +\Psi(c) \right) \dH{d-1}$  represents the  Ginzburg--Landau energy
accounting for interfacial effects and phase separation phenomena.
%describing energetic effects due to the chemical species $c$.
The function $\Psi(c)$ is a
double-well potential having two distinct global minima and $\Psi(c) =\frac 12 (1-c^2)^2 $ is a typical example. The contribution $\frac \gamma 2|\nabla_\Gamma c|^2$ ensures that interfacial energy effects are taken into account in the energy. Here, $\nabla_\Gamma$ is the surface gradient on the surface $\Gamma$. Furthermore, $\dH{d-1}$ denotes integration with respect to the $(d-1) $-dimensional surface measure.

The evolution equations we now consider consist of an $L_2$-gradient flow for the
evolution of the interface $\Gamma$ and a Cahn-Hilliard type evolution for the surface.
The Cahn--Hilliard evolution equation on an evolving surface as derived in \cite{CaetanoEll21}, see also Section
\ref{sec:model}, is given as
\begin{equation}
	\label{CH-surface1}
	\partial^{\bullet} c =\diver_\Gamma(m(c)\nabla_\Gamma\mu)+cVH.
	\end{equation}
Here, $\partial^\bullet$ denotes the normal time derivative, $\diver_\Gamma$ is the surface divergence,  $m$ is a positive mobility function, $\mu$ is the chemical potential and  $V$ is the normal velocity  of an evolving hypersurface $\Gamma=\Gamma(t)$. We remark that $VH = - \diver_\Gamma (V \nu)$
where $\nu$ is the normal to the interface. As $V \nu$ is the velocity of the interface the term $cVH$
 accounts for changes due to the geometric motion of the surface.
 This follows from a transport theorem for an evolving interface, see e.g.
\cite{DDE05, BarrettGN20, BDGP23}.
 The chemical potential $\mu$ is given as the variational derivative  of $\cE$ with respect to $c$.
In fact, we obtain, see \cite{ElliotStinner10,ElliotStinner10b,BarrettGN17},
\begin{equation}
	\label{CH-surface2}
	\mu = -\gamma \Delta_\Gamma c+\Psi'(c) + \frac 12 \alpha'(c) (H-\bar{H}(c))^2  - \alpha (c) (H-\bar{H}(c)) \bar{H}'(c) + G(\alpha^G)'(c) .
\end{equation}

The system is completed by the $L_2$-gradient flow of $\cE$ for  $\Gamma$, i.e., the normal velocity $V$ is given as the negative variational derivative of $\cE$ with respect to $\Gamma$. The resulting geometric evolution equation is, see \cite{ElliotStinner10,ElliotStinner10b,BarrettGN17},
 \begin{multline}
	\label{surfaceDiff}
	V =-\Delta_\Gamma\left(\alpha(c)(H-\bar{H}(c))\right)-|L|^2\alpha(c)(H-\bar{H}(c))+\frac{1}{2}\alpha(c)(H-\bar{H}(c))^2 H \\
	   -\diver_\Gamma \left((H\operatorname{Id}+L)\nabla_\Gamma\alpha^G(c)\right)
	+\gamma  ((\nabla_\Gamma c)^T  L \nabla_\Gamma c  ) +
	\left( \frac \gamma 2 | \nabla_\Gamma c|^2 +\Psi(c) \right) H-\mu cH.
\end{multline}
Here, $ L=-\nabla_\Gamma \nu$ is the Weingarten map.
We assume that
 $m,\alpha,\alpha^G,\bar{H}$ and $\Psi$ are given functions subject to the following assumptions.
\begin{itemize}
	\item[(R)] The constitutive functions $m,\alpha,\alpha^G,\bar{H}$ and $\Psi$ are assumed to be smooth ($m\in C^2(\R)$ and $\alpha,\alpha^G,\bar{H},\Psi\in C^4(\R)$ would suffice).
	\item[(P)] For all $s\in\R$ we assume   $\alpha(s)>0$ and $m(s)>0$.
\end{itemize}
The system \eqref{CH-surface1}--\eqref{CH-surface2} can be seen as an $H^{-1}$-gradient flow of $\cE$  with respect to  the variable $c$ and the equation \eqref{surfaceDiff} can be seen as an $L_2$-gradient flow with respect to $\Gamma$, see \cite{BDGP23, Garcke}. These facts imply that the system \eqref{CH-surface1}--\eqref{surfaceDiff} is thermodynamically consistent in the sense that the energy $\cE$ decreases in time.
In fact, we obtain
\begin{equation}\label{energdecay}
	\frac d{dt} \cE (\Gamma , c)= -\int_\Gamma \left(V^2 + m(c)
	|\nabla_\Gamma c|^2 \right) \dH{d-1}.
\end{equation}
For other $(L_{2},H^{-1})$--surface gradient flows leading to second order parabolic equations we refer to
\cite{AbelsBG23, AbelsBG23b, backofen2026}.

The remainder of the paper is organized as follows. In Section \ref{sec:model} we derive the system
\eqref{CH-surface1}--\eqref{surfaceDiff},  show that the total mass $\int_\Gamma c
\dH{d-1}$ is conserved and prove the energy decay property \eqref{energdecay}.
In Section \ref{sec:GS} we reformulate the governing equations over a given compact, smooth, fixed (embedded) reference surface $\Sigma$.
We then introduce a proper functional analytic setting in Section \ref{sec:FAsetting} and prove in Section \ref{Local well-posedness} the
existence of a strong local solution to \eqref{CH-surface1}--\eqref{surfaceDiff} having optimal $L_p$-$L_q$-regularity by using the theory of quasi-linear parabolic evolution equations in time-weighted $L_p$-spaces introduced in \cite{KPW10, LPW14}. We note that the system \eqref{CH-surface1}--\eqref{surfaceDiff} contains a term of the form
$\Delta_\Gamma (H^2)$, which leads to a contribution that is quadratic in third-order operators. Such terms ask for estimates of Gagliardo--Nirenberg type, and precisely this theory has been developed in \cite{LPW14}.
%This significantly complicates the analysis and makes .

\section{Derivation of the governing equations}\label{sec:model}

The evolution equations we want to study involve a balance law for a chemical species diffusing on the evolving
surface. To derive the balance law in a local form, we use a transport theorem for an evolving surface. In fact, for a $C^1$-function $f$ defined
on an evolving surface $\Gamma=\Gamma(t)$, it holds, see \cite{BarrettGN20} for a proof of the transport theorem and
other definitions and results in the context of evolving hypersurfaces,
\begin{equation}\label{transportheorem}
	\frac d{dt}\int_\Gamma f \dH{d-1} =\int_\Gamma (\partial^\bullet f -  f HV )  \dH{d-1},
\end{equation}
where we assume that points on the surface are advected in normal direction. The above identity also holds if we replace $\Gamma$ by
a sufficiently smooth  $\Lambda(t)\subset \Gamma(t)$ which is transported in normal direction.
The balance law for a quantity $c$ diffusing on $\Gamma$ now states that for all sufficiently smooth  surface $\Lambda(t)\subset \Gamma(t)$  it holds
that, see \cite{CaetanoEll21},
\begin{equation}
	\label{balancelaw}
	\frac d{dt}\int_\Lambda c \dH{d-1} = - \int_\Lambda \diver_\Gamma q \dH{d-1},
\end{equation}
where $q$ is the tangential mass flux on $\Gamma(t)$.
Combining the above and using the arbitrariness of
 $\Lambda$, we obtain
the balance law
\begin{equation}
	\label{balancelaw2}
	\partial^\bullet c -  c HV = -\diver_\Gamma q.
\end{equation}
Now \eqref{transportheorem} for $f=c$ gives with the help of
\eqref{balancelaw2} and the divergence theorem that the total mass is conserved.
%As in the original Cahn--Hilliard equation we now choose the flux $q$ as
%$$ q=-m(c) \nabla \mu, $$
%where $\mu$ is the chemical potential which as in the original Cahn--Hilliard model we plan to choose   as  the functional derivative of the energy $\cE$ with respect to $c$.

%FIRST WE DERIVE THE CAHN HILLIARD EQUAUION FROM BALANCE LAWS.
%
%THIS IN PARTCULAR GIVES THE TERM
%
%$$cVH$$
%
%FOR THE NORMAL VELOCITY WE TAKE THE
%$L_2$--GRADIENT FLOW
%To be able to compute functional derivative of the energy $\cE$ with respect to $c$ and the evolution of the surface as the $L_2$-gradient of the energy,
We now compute the derivative of the energy of a surface moving with normal velocity $V$, coupled to an evolving concentration field with normal time derivative
$\partial^\bullet c$.  We then want to choose the tangential mass flux $q$ and the normal velocity of the surface such that the energy cannot increase.
The derivative of the energy under such evolutions was computed in \cite{ElliotStinner10b,BarrettGN17} using the transport theorem \eqref{transportheorem} and evolution equations for geometric quantities, such as the metric and curvature on evolving surfaces; see also \cite{BarrettGN20}.
The variation of the  mean curvature part of the energy
$$\int_\Gamma \frac 12 \alpha (c) (H-\bar{H}(c))^2\dH{d-1}
$$
has been computed in Lemma 4.4 in \cite{ElliotStinner10b} and in (A.5) of \cite{BarrettGN17}. The variation of the Gaussian energy part
$$
\int_\Gamma \alpha^G(c) G \dH{d-1}
$$
can be computed with the help of the formula $ G=\frac 12(H^2- |L|^2)$, see
Lemma 4.5 in \cite{ElliotStinner10b} and (A.17) in \cite{BarrettGN17}.
 Finally, the variational derivative of the  Ginzburg--Landau energy
 $$\int_\Gamma
 \left( \frac \gamma 2| \nabla_\Gamma c|^2 +\Psi(c)\right)  \dH{d-1}$$
 has been computed in Lemma 4.7 in\cite{ElliotStinner10b} and in
 (A.6) in \cite{BarrettGN17}.
Using the above mentioned  results
%, the equations \eqref{CH-surface1}--\eqref{surfaceDiff} and integration by parts
we get
\begin{eqnarray*}
	\frac d{dt} \cE (\Gamma , c)&=& \int_\Gamma \partial^\bullet c \Big(-\gamma \Delta_\Gamma c+\Psi'(c) + \frac 12 \alpha'(c) (H-\bar{H}(c))^2  - \alpha (c) (H-\bar{H}(c)) \bar{H}'(c) \\
	&& \qquad +G(\alpha^G)'(c) \Big)
	+V\Big(\Delta_\Gamma\left(\alpha(c)(H-\bar{H}(c))\right)+|L|^2\alpha(c)(H-\bar{H}(c))
	\\&&\qquad-\frac{1}{2}\alpha(c)(H-\bar{H}(c))^2 H
	\;\diver_\Gamma \left((H\operatorname{Id}+L)\nabla_\Gamma\alpha^G(c)\right)
	-	\\&& \qquad\gamma  ((\nabla_\Gamma c)^T  L \nabla_\Gamma c  ) -
	( \frac \gamma 2 | \nabla_\Gamma c|^2 +\Psi(c) ) H \Big)\dH{d-1}.
	\end{eqnarray*}
In order to obtain evolution laws that are thermodynamically consistent, we  choose the tangential mass flux $q$ and the normal velocity $V$ such that the energy does not increase.	Using the mass balance law $\partial^\bullet c -  c HV = -\diver_\Gamma q$,
and setting $\mu =-\gamma \Delta_\Gamma c+\Psi'(c) + \frac 12 \alpha'(c) (H-\bar{H}(c))^2  - \alpha (c) (H-\bar{H}(c)) \bar{H}'(c) + G(\alpha^G)'(c)$ we get using integration by parts on surfaces
	\begin{eqnarray*}
		\frac d{dt} \cE (\Gamma , c)
		&=&
		\int_\Gamma
		 q\cdot \nabla_\Gamma \mu\dH{d-1}
		% (c HV +\diver_\Gamma q ) \mu
%		
%		 \Big(-\gamma \Delta_\Gamma c+\Psi'(c) + \frac 12 \alpha'(c) (H-\bar{H}(c))^2  - \alpha (c) (H-\bar{H}(c)) \bar{H}'(c) \\
%		&& \qquad +G(\alpha^G)'(c) \Big)
	\\	&&+\int_\Gamma V\Big( cH\mu+ \Delta_\Gamma\left(\alpha(c)(H-\bar{H}(c))\right)+|L|^2\alpha(c)(H-\bar{H}(c))
		\\&&-\frac{1}{2}\alpha(c)(H-\bar{H}(c))^2 H
		-\diver_\Gamma \left((H\operatorname{Id}+L)\nabla_\Gamma\alpha^G(c)\right)
			\\&& \qquad-\gamma  ((\nabla_\Gamma c)^T  L \nabla_\Gamma c  ) -
		( \frac \gamma 2 | \nabla_\Gamma c|^2 +\Psi(c) ) H \Big)\dH{d-1} .
	\end{eqnarray*}
	Defining the normal velocity as in
	\eqref{surfaceDiff} and setting  $q=-m(c) \nabla \mu$
		with a mobility  $m\geq 0$ we obtain
%	\int_\Gamma \partial^\bullet c\mu  - (V+{\mu cH})V\dH{d-1}\\
\begin{eqnarray*}
\frac d{dt} \cE (\Gamma , c)
%	&	=&\int_\Gamma\mu (\diver_\Gamma(m(c)\nabla_\Gamma\mu)+cVH)
%	- V^2-{\mu cH}V\dH{d-1}\\
	&=&-\int_\Gamma \left( m(c)
	|\nabla_\Gamma \mu|^2 + V^2 \right) \dH{d-1}\leq 0.
\end{eqnarray*}
This shows that the system \eqref{CH-surface1}--\eqref{surfaceDiff} is thermodynamically consistent in the sense that the free energy is non-increasing in time. A gradient flow
interpretation of the system \eqref{CH-surface1}--\eqref{surfaceDiff} can be derived as in \cite{AbelsBG23}.

\section{Geometric setting}\label{sec:GS}

In the following, we assume that the evolving surface $\Gamma(t)$ is represented as the graph of a height function $\rho$ over a given compact, smooth, fixed (embedded) reference surface $\Sigma$. That is,
$$\Gamma(t) = \{\psi_\rho(p) := p + \rho(t,p)\nu_\Sigma(p) : p \in \Sigma\},\quad t \ge 0.$$
Using coordinates on $\Sigma$, we  obtain the following transformed differential operators for $u_* := u \circ \psi_\rho$ and $F_* := F \circ \psi_\rho$:
$$\nabla_\rho u_* = (\nabla_\Gamma u)\circ\psi_\rho = P_\Gamma M_0(\rho)\nabla_\Sigma u_*,$$
$$\diver_\rho F_* = (\diver_\Gamma F)\circ\psi_\rho = \operatorname{tr}\left[P_\Gamma M_0(\rho)\nabla_\Sigma F_*\right],$$
$$\Delta_\rho u_* = (\Delta_\Gamma u)\circ\psi_\rho = M_0(\rho)P_\Gamma M_0(\rho) : \nabla_\Sigma^2 u_* + (b(\nabla_\Sigma^k\rho) | \nabla_\Sigma u_*),$$
where $k \in \{0, 1, 2\}$. Furthermore, we have
$$\nu_\rho := \nu_\Gamma\circ\psi_\rho = \beta(\rho)\big(\nu_\Sigma - a(\rho)\big),$$
$$V_\rho := V\circ\psi_\rho = \beta(\rho)\partial_t\rho,$$
$$H_\rho := H\circ\psi_\rho = -\diver_\rho\nu_\rho,$$
$$L_\rho := L\circ\psi_\rho = -\nabla_\rho\nu_\rho,$$
and the transformed normal time derivative is given by
$$\partial^{\bullet}u_* = \partial_t u_* - \partial_t\rho\, (\nu_\Sigma |\nabla_\rho u_*).$$
%where $u_* := u \circ \psi_\rho$ and $F_* := F \circ \psi_\rho$. We
Here, we use notation as in \cite[Section 2.2 \& 2.5]{PruSim16}
%. In particular,
and in particular, for the above the following relations need to hold:
$$M_0(\rho) = (I - \rho L_\Sigma)^{-1},\quad a(\rho) = M_0(\rho)\nabla_\Sigma \rho,\quad \beta(\rho) = (1 + |a(\rho)|_2^2)^{-1/2},$$
provided that $\|\rho\|_{L_\infty(\Sigma)} < r$ for some sufficiently small $r > 0$ (see \cite[(2.43)]{PruSim16}), where $L_\Sigma = -\nabla_\Sigma\nu_\Sigma$ denotes the Weingarten tensor (see \cite[Section 2.1 \& 2.2]{PruSim16}).

Now, setting $(v, \eta) := (c, \mu) \circ \psi_\rho$, the transformed system on $\Sigma$ reads:
\begin{equation}
    \label{CH-surface_transf}
    \begin{aligned}
    \partial_t v &= \diver_\rho\big(m(v)\nabla_\rho\eta\big) + \partial_t\rho\left(v\beta(\rho) H_\rho + (\nu_\Sigma|\nabla_\rho v)\right)&&\text{on}\;\;\Sigma,\\
    \eta &= -\gamma\Delta_\rho v + \Psi'(v)\\
    &\quad+ \frac 12 \alpha'(v) (H_\rho-\bar{H}(v))^2  - \alpha (v) (H_\rho-\bar{H}(v)) \bar{H}'(v) + G(\alpha^G)'(v) &&\text{on}\;\; \Sigma,
    \end{aligned}
\end{equation}
and
\begin{multline}
    \label{surfaceDiff_transf}
    \beta(\rho)\partial_t \rho = -\Delta_\rho\big(\alpha(v)(H_\rho-\bar{H}(v))\big) - |L_\rho|^2\alpha(v)(H_\rho-\bar{H}(v))\\
     +\frac{1}{2}\alpha(v)(H_\rho-\bar{H}(v))^2 H_\rho - \diver_\rho\left((H_\rho\operatorname{Id} + L_\rho)\nabla_\rho\alpha^G(v)\right)\\
     +\gamma  ((\nabla_\rho v)^T  L_\rho \nabla_\rho v  ) +
	\left( \frac \gamma 2 | \nabla_\rho v|^2 +\Psi(v) \right) H_\rho-\eta v H_\rho\;\;\text{on}\; \Sigma,
\end{multline}
with initial values $(v(0), \rho(0)) = (v_0, \rho_0)$.

Finally,  for $\rho = 0$ (i.e., on $\Sigma$ itself), the operators reduce to their classical forms: $\nabla_0 = \nabla_\Sigma$ (surface gradient), $\diver_0 = \diver_\Sigma$ (surface divergence), $\Delta_0 = \Delta_\Sigma$ (Laplace-Beltrami operator), $\nu_0 = \nu_\Sigma$, $H_0 = H_\Sigma := -\diver_\Sigma \nu_\Sigma$ (mean curvature), and $L_0 = L_\Sigma$.

\section{Functional analytic setting}\label{sec:FAsetting}

\subsection{Function spaces}

According to \cite[Section 2.2]{PruSim16}, both $-\Delta_\rho$ and $-H_\rho$ are second-order, strongly elliptic operators acting on functions defined on $\Sigma$. Consequently, the highest-order terms in the equations for $v$ and $\rho$ are given by $\partial_t v$, $\Delta_\rho^2 v$, $\partial_t\rho$, and $\Delta_\rho H_\rho$.

For $1<p,q<\infty$ and $\mu\in (1/p,1]$, we select $L_{p,\mu}((0,T);L_q(\Sigma))$ as the base space for equations \eqref{CH-surface_transf} and \eqref{surfaceDiff_transf}. Here, for any UMD-space $E$, the time-weighted function spaces $L_{p,\mu}((0,T);E)$ are defined as
$$
L_{p,\mu}((0,T);E) := \{u\colon (0,T)\to E \mid [t\mapsto t^{1-\mu} u(t)] \in L_p((0,T);E)\},
$$
equipped with the norm
$$
\|u \|_{L_{p,\mu}((0,T);E)} := \| [t\mapsto t^{1-\mu} u(t)] \|_{L_p((0,T);E)}.
$$
Given $k \in \mathbb{N}_0$, the corresponding weighted Sobolev spaces are defined by
$$
H_{p,\mu}^{k}((0,T);E) := \{u \in W_{1,\mathrm{loc}}^{k}((0,T);E) \mid u^{(j)} \in L_{p,\mu}((0,T);E), \enspace j
\in \{0,\dots,k\}\},
$$
with the natural norms. For properties and applications of these weighted function spaces, we refer to \cite{PruSim16} or the survey \cite{Wil23}.

Since both $\Delta_\rho^2 v$ and $\Delta_\rho H_\rho$ are operators of order four, the optimal regularity class for $v$ and $\rho$ is
$$
H_{p,\mu}^1((0,T);L_q(\Sigma))\cap L_{p,\mu}((0,T);H_q^4(\Sigma)).
$$
In view of the term
$$
\partial_t\rho\left(v\beta(\rho) H_\rho+(\nu_\Sigma|\nabla_\rho v)\right)
$$
in \eqref{CH-surface_transf}, it is necessary for
$$
v\beta(\rho) H_\rho+(\nu_\Sigma|\nabla_\rho v)
$$
to act as a multiplier in $L_{p,\mu}((0,T);L_q(\Sigma))$, i.e.,
$$
v\beta(\rho) H_\rho+(\nu_\Sigma|\nabla_\rho v) \in L_\infty((0,T);L_\infty(\Sigma)).
$$
This condition is certainly fulfilled if
$$
v,\rho\in C([0,T];C^2(\Sigma)).
$$
To achieve this, note that \cite[Theorem 3.4.8]{PruSim16} yields the embedding
$$
H_{p,\mu}^1((0,T);L_q(\Sigma))\cap L_{p,\mu}((0,T);H_q^4(\Sigma))\hookrightarrow C([0,T];B_{qp}^{4\mu-4/p}(\Sigma)),
$$
where
$$(L_q(\Sigma),H_q^4(\Sigma))_{\mu-1/p,p}=B_{qp}^{4\mu-4/p}(\Sigma)$$
is the trace space of $v,\rho$ and $B_{qp}^s(\Sigma)$ denotes a Besov space. From now on, let us assume that
$p,q\in (1,\infty)$ and $\mu\in (1/p,1]$ satisfy
\begin{equation}\label{eq:Assump_pq}
\frac{4}{p}+\frac{n}{q}<2\quad\text{and}\quad 2+\frac{4}{p}+\frac{n}{q}<4\mu.
\end{equation}
Under these assumptions, it follows that
\begin{equation}\label{eq:SobEmb1}
B_{qp}^{4\mu-4/p}(\Sigma)\hookrightarrow C^2(\Sigma),
\end{equation}
so that in particular
$$
v,\rho\in C([0,T];C^2(\Sigma)).
$$
We note that the embedding \eqref{eq:SobEmb1} follows from \cite[Theorem 4.6.1]{Tri78} in combination with \cite[Theorem 4.4]{Ama25}.

For the remainder, we define the spaces $X_0 = L_q(\Sigma)$, $X_1 = H_q^4(\Sigma)$, and
$$
X_{\gamma, \mu} := (X_0, X_1)_{\mu-1/p, p} = B_{qp}^{4\mu-4/p}(\Sigma).
$$
Furthermore, we set
$$
V_\mu := \{\rho \in X_{\gamma, \mu} \mid \|\rho\|_{L_\infty(\Sigma)} < r \}.
$$

\subsection{Reformulation of the problem}

In this section, we reformulate \eqref{CH-surface_transf} and \eqref{surfaceDiff_transf} in an abstract framework. We begin with equation \eqref{surfaceDiff_transf} for $\rho$ and introduce $\mathcal{F}_h(v,\rho)$ as the right-hand side of \eqref{surfaceDiff_transf} divided by $\beta(\rho)$, so that $\partial_t\rho = \mathcal{F}_h(v,\rho)$. For local coordinates $(U, \varphi)$, where $U \subset \Sigma$ is open, we define $(v_*, \rho_*) := \varphi_*(v,\rho) := (v,\rho) \circ \varphi^{-1}$.

\begin{proposition}\label{pro:quasi_rho}
Let $1<p<\infty$, $\mu\in (1/p,1]$ satisfy \eqref{eq:Assump_pq} and assume (R).
For $(v, \rho) \in X_1 \times (X_1 \cap V_\mu)$, it holds that
$$
\mathcal{F}_h(v,\rho) = A_h(v,\rho)\rho + \text{lower order terms},
$$
where $A_h:X_{\gamma,\mu}\times V_\mu \to \mathcal{L}(X_1, X_0)$, and in local coordinates its representation is given by
$$
\varphi_*A_h(v,\rho)\hat{\rho} = \sum_{3 \le |\sigma| \le 4} a_\sigma^{h}(v_*, \partial v_*, \rho_*, \partial \rho_*, \partial^2 \rho_*) \partial^\sigma \hat{\rho}_*.
$$
Here, $a_\sigma^{h}$ is a smooth function, and the leading order term reads
$$
\sum_{|\sigma|=4} a_\sigma^{h}(v_*,\rho_*,\partial \rho_*) \partial^\sigma \hat{\rho}_* = -\frac{1}{\beta(\rho_*)}\alpha(v_*) \sum_{i,j,k,\ell} a^{k\ell}(\rho_*,\partial \rho_*) c^{ij}(\rho_*,\partial \rho_*) \partial_i \partial_j \partial_k \partial_\ell \hat{\rho}_*,
$$
with positive definite matrices $(a^{k\ell})$ and $(c^{ij})$. The lower order terms involve derivatives of $v$ and $\rho$ of order at most $2$.
\end{proposition}

\begin{proof}
The third- and fourth-order terms in $\mathcal{F}_h$ originate from $\Delta_\rho(\alpha(v) H_\rho)$ and $\operatorname{div}_\rho (H_\rho \operatorname{Id} + L_\rho)$. All remaining terms in $\mathcal{F}_h$ involve at most second-order derivatives of $v$ and $\rho$.

For $\rho \in X_1 \cap V_\mu$, we have
$$
\Delta_\rho(\alpha(v) H_\rho) = H_\rho\, \Delta_\rho \alpha(v) + 2\, \nabla_\rho \alpha(v) \cdot \nabla_\rho H_\rho + \alpha(v)\, \Delta_\rho H_\rho.
$$
The term $H_\rho\, \Delta_\rho \alpha(v)$ is second order in both $v$ and $\rho$. We now consider the terms $\nabla_\rho \alpha(v) \cdot \nabla_\rho H_\rho$ and $\alpha(v)\, \Delta_\rho H_\rho$.

For readability, the subscript $*$ will be omitted when working in local coordinates. According to \cite[Section 2.2.3 \& 2.2.5]{PruSim16}, we have
$$
\Delta_\rho = a^{ij}(\rho, \partial\rho) \partial_i \partial_j + b^k(\rho, \partial\rho, \partial^2 \rho) \partial_k,
$$
and
$$
H_\rho = c^{ij}(\rho, \partial \rho) \partial_i \partial_j \rho + h(\rho, \partial \rho),
$$
where $(a^{ij})$ and $(c^{ij})$ are symmetric and positive definite and $b^k$ is linear in $\partial^2\rho$. A direct calculation yields
$$
\begin{aligned}
&\partial_k\partial_\ell \big( c^{ij}(\rho, \partial \rho) \partial_i \partial_j \rho + h(\rho, \partial \rho) \big) \\
&= c^{ij}\, \rho_{ijk\ell}
+ \bigl( c^{ij}_\rho \rho_k + c^{ij}_{\rho_m} \rho_{mk} \bigr) \rho_{ij\ell}
+ \bigl( c^{ij}_\rho \rho_\ell + c^{ij}_{\rho_m} \rho_{m\ell} \bigr) \rho_{ijk} \\
&\quad + \Bigl(
c^{ij}_{\rho\rho} \rho_k \rho_\ell
+ c^{ij}_\rho \rho_{k\ell}
+ c^{ij}_{\rho\,\rho_m} (\rho_{m\ell} \rho_k + \rho_{mk} \rho_\ell)
+ c^{ij}_{\rho_m\rho_n} \rho_{mk} \rho_{n\ell}
+ c^{ij}_{\rho_m} \rho_{mk\ell}
\Bigr) \rho_{ij} \\
&\quad + h_{\rho\rho}\, \rho_k \rho_\ell + h_\rho\, \rho_{k\ell} + h_{\rho\,\rho_n} \rho_{nk} \rho_\ell \\
&\quad + h_{\rho_m\rho} \rho_k \rho_{m\ell}
+ h_{\rho_m\rho_n} \rho_{nk} \rho_{m\ell}
+ h_{\rho_m} \rho_{mk\ell}
\end{aligned}
$$
and
$$
\partial_k \left( c^{ij}(\rho, \partial \rho)\, \partial_i \partial_j \rho + h(\rho, \partial \rho) \right)
= \bigl( c^{ij}_\rho \rho_k + c^{ij}_{\rho_m} \rho_{mk} \bigr) \rho_{ij}
+ c^{ij} \rho_{ijk} + h_\rho \rho_k + h_{\rho_m} \rho_{mk}.
$$
Thus,
$$
\begin{aligned}
\Delta_\rho H_\rho &= a^{k\ell}\bigl[ c^{ij} \rho_{ijk\ell}
+ \bigl( c_\rho^{ij} \rho_k + c_{\rho_m}^{ij} \rho_{mk} \bigr) \rho_{ij\ell}
+ \bigl( c_\rho^{ij} \rho_\ell + c_{\rho_m}^{ij} \rho_{m\ell} \bigr) \rho_{ijk} \\
&\quad + (c_{\rho\rho}^{ij} \rho_k \rho_\ell + c_\rho^{ij} \rho_{k\ell}
+ c_{\rho\,\rho_m}^{ij} ( \rho_{m\ell}\rho_k + \rho_{mk}\rho_\ell )
+ c_{\rho_m\rho_n}^{ij} \rho_{mk}\rho_{n\ell}
+ c_{\rho_m}^{ij} \rho_{mk\ell} ) \rho_{ij} \\
&\quad + h_{\rho\rho} \rho_k \rho_\ell + h_\rho \rho_{k\ell}
+ h_{\rho\,\rho_n} \rho_{nk} \rho_\ell \\
&\quad + h_{\rho_m\rho} \rho_k \rho_{m\ell}
+ h_{\rho_m\rho_n} \rho_{nk} \rho_{m\ell}
+ h_{\rho_m} \rho_{mk\ell} \bigr] \\
&\quad + b^k \bigl[ (c_\rho^{ij} \rho_k + c_{\rho_m}^{ij} \rho_{mk}) \rho_{ij}
+ c^{ij} \rho_{ijk} + h_\rho \rho_k + h_{\rho_m} \rho_{mk} \bigr]
\end{aligned}
$$
To simplify notation, we have set $\rho_i := \partial_i \rho$, $\rho_{ij} := \partial_i\partial_j\rho$, $h_\rho := \partial_\rho h$, $c_{\rho\rho_m}^{ij} := \partial_\rho \partial_{\rho_m} c^{ij}$, and so on.

Next, the transformed surface gradient $\nabla_\rho$ in local coordinates is given by
$$
\nabla_\rho = g_\Gamma^{k\ell} \partial_\ell \psi_\rho \partial_k,
$$
so that
\begin{equation}\label{eq:nablaHrho}
\nabla_\rho H_\rho = g_\Gamma^{k\ell} \left[ (c_\rho^{ij} \rho_k + c_{\rho_m}^{ij} \rho_{mk} ) \rho_{ij} + c^{ij} \rho_{ijk} + h_\rho \rho_k + h_{\rho_m} \rho_{mk} \right] \partial_\ell \psi_\rho
\end{equation}
and
$$
\nabla_\rho \alpha(v) = g_\Gamma^{k\ell} \alpha'(v) \partial_k v \partial_\ell \psi_\rho.
$$
It is important to observe that the third- and fourth-order derivatives of $\rho$ appear linearly in both $\Delta_\rho H_\rho$ and $\nabla_\rho H_\rho$. The leading term of fourth-order in $\Delta_\rho H_\rho$ is
$$
a^{k\ell}(\rho, \partial\rho) c^{ij}(\rho, \partial\rho) \partial_i \partial_j \partial_k \partial_\ell \rho.
$$
Analogous reasoning shows that $\operatorname{div}_\rho (H_\rho \operatorname{Id} + L_\rho)$ is of third order in $\rho$ and first order in $v$, with third-order derivatives of $\rho$ appearing linearly.

This demonstrates that
$$
\mathcal{F}_h(v,\rho) = A_h(v,\rho)\rho + \text{lower order terms},
$$
where $A_h:X_{\gamma,\mu}\times V_\mu \to \mathcal{L}(X_1, X_0)$ satisfies the representations above, and the lower order terms involve only derivatives of $v$ and $\rho$ of order at most $2$.
\end{proof}

\noindent
By Proposition \ref{pro:quasi_rho},
$$
F_h(v,\rho) := \mathcal{F}_h(v,\rho) - A_h(v,\rho)\rho
$$
contains only derivatives of $v$ and $\rho$ up to second order, and $F_h$ depends smoothly on these variables. Therefore, in local coordinates,
$$
\varphi_* F_h(v,\rho) = f_h(v_*, \partial v_*, \partial^2 v_*, \rho_*, \partial \rho_*, \partial^2 \rho_*),
$$
for a suitable smooth function $f_h$. Hence the mapping $F_h: X_{\gamma,\mu} \times V_\mu \to X_0$ is well defined (we recall that $\Sigma$ is compact and $X_{\gamma,\mu} \hookrightarrow C^2(\Sigma)$ holds by \eqref{eq:Assump_pq}).

Thus, we may write the equation \eqref{surfaceDiff_transf} in condensed form as
\begin{equation}\label{eq:rho_condensed}
\partial_t \rho = A_h(v,\rho)\rho + F_h(v,\rho),
\end{equation}
Let us now consider the transformed equation \eqref{CH-surface_transf} and define $$\mathcal{F}_c(v,\rho) := \operatorname{div}_\rho \left( m(v) \nabla_\rho \eta \right ),$$
where
\begin{align*}
\eta &= -\gamma\Delta_\rho v + \Psi'(v)+ \frac 12 \alpha'(v) (H_\rho-\bar{H}(v))^2  \\
&\quad - \alpha (v) (H_\rho-\bar{H}(v)) \bar{H}'(v) + G(\alpha^G)'(v).
\end{align*}
\begin{proposition}\label{pro:quasi_v}
Let $1<p<\infty$, $\mu\in (1/p,1]$ satisfy \eqref{eq:Assump_pq} and assume (R).
For $(v, \rho) \in X_1 \times (X_1 \cap V_\mu)$, it holds that
$$
\mathcal{F}_c(v,\rho) = A_c(v,\rho)v + B_h(v,\rho)\rho +Q(v,\rho)+ \text{lower order terms},
$$
where $A_c, B_h: X_{\gamma,\mu} \times V_\mu \to \mathcal{L}(X_1, X_0)$ have the following local coordinate representations:
$$
\varphi_* A_c(v,\rho) \hat{v} = \sum_{3 \le |\sigma| \le 4} a_\sigma^{c}(v_*, \partial v_*, \rho_*, \partial \rho_*, \partial^2 \rho_*) \partial^\sigma \hat{v}_*,
$$
and
$$
\varphi_* B_h(v,\rho) \hat{\rho} = \sum_{3 \le |\sigma| \le 4} b_\sigma^{h}(v_*, \partial v_*,\partial^2 v_* , \rho_*, \partial \rho_*, \partial^2 \rho_*) \partial^\sigma \hat{\rho}_*,
$$
with smooth functions $a_\sigma^c$ and $b_\sigma^h$. The leading order term of $A_c$ is given by
$$
\sum_{|\sigma|=4} a_\sigma^c(v_*, \rho_*, \partial \rho_*) \partial^\sigma \hat{v}_* = -\gamma m(v_*) \sum_{i,j,k,\ell} a^{k\ell}(\rho_*,\partial \rho_*) a^{ij}(\rho_*,\partial \rho_*) \partial_i \partial_j \partial_k \partial_\ell \hat{v}_*,
$$
with a positive definite matrix $(a^{ij})$. Moreover, for $X_\beta:=(X_0,X_1)_{\beta,p}$ and $\beta\in (3/4+n/(8q),1)$, the mapping $Q:X_\beta\times (V_\mu\cap X_\beta)\to X_0$
has the local coordinate representation
$$\varphi_*Q(v,\rho)=\sum_{|\sigma|=3}r_{\sigma}^h(v_*,\rho_*,\partial\rho_*)(\partial^\sigma\rho_*)^2,$$
for some smooth function $r_\sigma^h$.

The lower order terms involve derivatives of $v$ and $\rho$ of order at most $2$.
\end{proposition}

\begin{proof}
For $\rho \in X_1 \cap V_\mu$, we obtain
$$
\operatorname{div}_\rho \left(m(v) \nabla_\rho \eta \right) = \nabla_\rho m(v) \cdot \nabla_\rho \eta + m(v) \operatorname{div}_\rho \nabla_\rho \eta = \nabla_\rho m(v) \cdot \nabla_\rho \eta + m(v) \Delta_\rho \eta,
$$
where
\begin{multline*}
\eta = -\gamma\Delta_\rho v + \Psi'(v)+ \frac 12 \alpha'(v) (H_\rho-\bar{H}(v))^2  \\ - \alpha (v) (H_\rho-\bar{H}(v)) \bar{H}'(v) + G(\alpha^G)'(v).
\end{multline*}
The third- and fourth-order terms in $\mathcal{F}_c$ arise from $\nabla_\rho \Delta_\rho v$, $\nabla_\rho H_\rho^j$, $\Delta_\rho H_\rho^j$ and $\Delta_\rho^2 v$, where $j\in\{1,2\}$. All other terms in $\mathcal{F}_c$ involve at most second-order derivatives of $v$ and $\rho$. Let us note that the terms $\nabla_\rho H_\rho$ as well as $\Delta_\rho H_\rho$ have already been investigated in the proof of Proposition \ref{pro:quasi_rho}.

We will now consider the terms $\nabla_\rho\Delta_\rho v$ and $\Delta_\rho^2 v$.
Suppressing the $*$ subscript for readability, we have:
$$
\begin{aligned}
&\partial_\ell \left( a^{ij}(\rho, \partial \rho)\, \partial_i \partial_j v
+ b^k(\rho, \partial \rho, \partial^2 \rho)\, \partial_k v \right) \\
&= \left[ \frac{\partial a^{ij}}{\partial \rho}\, \partial_\ell \rho + \frac{\partial a^{ij}}{\partial (\partial_k \rho)}\, \partial_\ell \partial_k \rho \right] \partial_i \partial_j v
+ a^{ij} \partial_\ell \partial_i \partial_j v \\
&\quad + \left[ \frac{\partial b^k}{\partial \rho} \partial_\ell \rho + \frac{\partial b^k}{\partial (\partial_m \rho)} \partial_\ell \partial_m \rho + \frac{\partial b^k}{\partial (\partial_m \partial_n \rho)} \partial_\ell \partial_m \partial_n \rho \right] \partial_k v \\
&\quad + b^k \partial_\ell \partial_k v,
\end{aligned}
$$
showing that $\nabla_\rho \Delta_\rho v$ depends only linearly on third-order derivatives of $v$ and $\rho$. Next, we compute
$$
\begin{aligned}
&\partial_m\partial_n\Big(
    a^{ij}(\rho,\partial\rho)\,\partial_i\partial_jv
    + b^k(\rho,\partial\rho,\partial^2\rho)\,\partial_k v
\Big) \\
&= \Big[\partial_m\partial_n\, a^{ij}(\rho,\partial\rho)\Big]\, \partial_i\partial_j v
  + \Big[\partial_m\, a^{ij}(\rho,\partial\rho)\Big]\, \partial_n\partial_i\partial_j v\\
&\quad + \Big[\partial_n\, a^{ij}(\rho,\partial\rho)\Big]\, \partial_m\partial_i\partial_j v
  + a^{ij}(\rho,\partial\rho)\, \partial_m\partial_n\partial_i\partial_j v \\
& \quad
  + \Big[\partial_m\partial_n\, b^k(\rho,\partial\rho,\partial^2\rho)\Big]\, \partial_k v
  + \Big[\partial_m\, b^k(\rho,\partial\rho,\partial^2\rho)\Big]\, \partial_n\partial_k v\\
&\quad  + \Big[\partial_n\, b^k(\rho,\partial\rho,\partial^2\rho)\Big]\, \partial_m\partial_k v
  + b^k(\rho,\partial\rho,\partial^2\rho)\, \partial_m\partial_n\partial_k v
\end{aligned}
$$
It follows that the leading order contribution to $\Delta_\rho^2 v$ with respect to $v$ in local coordinates is
$$
a^{k\ell}(\rho, \partial \rho) a^{ij}(\rho, \partial \rho) \partial_k \partial_\ell \partial_i \partial_j v,
$$
while the leading order term of $\Delta_\rho^2 v$ with respect to $\rho$ originates from
$$
\Big[\partial_m\partial_n\, b^k(\rho,\partial\rho,\partial^2\rho)\Big]\partial_k v,
$$
which reads (in local coordinates)
$$
\left[\frac{\partial b^k}{\partial (\partial_r\partial_s\rho)}\, \partial_m\partial_n\partial_r\partial_s \rho\right]\partial_k v.
$$
From the above, it is clear that third-order terms in $v$ and $\rho$ appear at most linearly in $\Delta_\rho^2 v$, since $b^k$ is linear in $\partial^2\rho$.

In a last step, we treat the terms $\nabla_\rho H_\rho^2$ and $\Delta_\rho H_\rho^2$. Since $\nabla_\rho H_\rho^2=2 H_\rho\nabla_\rho H_\rho$, it follows from the proof of Proposition \ref{pro:quasi_rho} that $\nabla_\rho H_\rho^2$ is of order 3 with third-order derivatives appearing at most linearly, since $H_\rho$ is of order two. Next, we compute
$$\Delta_\rho H_\rho^2=2 H_\rho\Delta_\rho H_\rho+2|\nabla_\rho H_\rho|^2.$$
Again, by the proof of Proposition \ref{pro:quasi_rho}, the third- and fourth-order derivatives of $\rho$ appear only linearly in $\Delta_\rho H_\rho$. By \eqref{eq:nablaHrho}, it holds that
$$\varphi_*|\nabla_\rho H_\rho|^2=\sum_{|\sigma|=3}r_{\sigma,1}^h(\rho_*,\partial\rho_*)(\partial^\sigma\rho_*)^2
+\sum_{|\sigma|=3} r_{\sigma,2}^h(\rho_*, \partial \rho_*, \partial^2 \rho_*) \partial^\sigma\rho_*+l.o.t.$$
for some smooth functions $r_{\sigma,j}^h$, $j\in\{1,2\}$.
For $X_\beta=(X_0,X_1)_{\beta,p}=B_{qp}^{4\beta}(\Sigma)$, $\beta\in (0,1)$, and $4\beta-n/q>3-n/2q$, it holds that
\begin{equation}\label{eq:SobEmb2}
B_{qp}^{4\beta}(\Sigma)\hookrightarrow H_{2q}^3(\Sigma),
\end{equation}
hence
$$\|(\nabla_\Sigma^3\rho)^2\|_{X_0}\lesssim\|\rho\|_{H_{2q}^3(\Sigma)}^2\lesssim \|\rho\|_{X_\beta}^2.$$
We note that the above conditions on $\beta$ are equivalent to $\beta\in (3/4+n/(8q),1)$ and the interval is not empty by \eqref{eq:Assump_pq}. Furthermore, the embedding \eqref{eq:SobEmb2} follows again from \cite[Theorem 4.6.1]{Tri78} in combination with \cite[Theorem 4.4]{Ama25}.

Summarizing, we have shown that
$$
\mathcal{F}_c(v,\rho) = A_c(v,\rho)v + B_h(v,\rho)\rho +Q(v,\rho)+ \text{lower order terms},
$$
where $A_c,B_h$ and $Q$ satisfy the representations given above, and the lower order terms contain only derivatives of $v$ and $\rho$ of order at most $2$.

%\begin{align*}
%\Delta_\rho\left(\alpha'(v)(H_\rho-\bar{H}(v))^2\right)&=(H_\rho-\bar{H}(v))^2\Delta_\rho \alpha'(v)\\
%&+2\nabla_\rho\alpha'(v)\cdot\nabla_\rho (H_\rho-\bar{H}(v))^2+\alpha'(v)\Delta_\rho (H_\rho-\bar{H}(v))^2.
%\end{align*}
%The first term on the right side is purely of second order in $v$ and $\rho$. Since
%$$\nabla_\rho (H_\rho-\bar{H}(v))^2=2(H_\rho-\bar{H}(v))\nabla_\rho (H_\rho-\bar{H}(v)),$$
%the second term contains of $v$
\end{proof}

\noindent
According to Proposition \ref{pro:quasi_v},
$$
F_c(v,\rho) := \mathcal{F}_c(v,\rho) - A_c(v,\rho)v - B_h(v,\rho)\rho-Q(v,\rho)
$$
depends smoothly on derivatives of $v$ and $\rho$ up to order $2$. Hence, in local coordinates,
$$
F_{c}(v,\rho) = f_c(v, \partial v, \partial^2 v, \rho, \partial \rho, \partial^2 \rho)
$$
for some smooth function $f_c$. Thus, the mapping $F_{c}: X_{\gamma,\mu} \times V_\mu \to X_0$ is well defined (again, since $\Sigma$ is compact and $X_{\gamma,\mu} \hookrightarrow C^2(\Sigma)$ by \eqref{eq:Assump_pq}).

Lastly, let us address the term
$$
\partial_t\rho\left( v\beta(\rho) H_\rho + (\nu_\Sigma | \nabla_\rho v) \right)
$$
in \eqref{CH-surface_transf}. Note that $G(v,\rho) := \left( v\beta(\rho) H_\rho + (\nu_\Sigma | \nabla_\rho v) \right)$ includes derivatives of $\rho$ up to order $2$ and derivatives of $v$ up to order $1$, so the mapping $G: X_{\gamma,\mu} \times V_\mu \to X_0$ is well defined by \eqref{eq:Assump_pq} and by the compactness of $\Sigma$.

Substituting $\partial_t\rho$ using equation \eqref{eq:rho_condensed}, we obtain
$$
\partial_t\rho \left( v\beta(\rho) H_\rho + (\nu_\Sigma | \nabla_\rho v) \right) = G(v,\rho) \left( A_h(v,\rho)\rho + F_h(v,\rho) \right).
$$
Therefore, we can reformulate \eqref{CH-surface_transf} and \eqref{surfaceDiff_transf} in condensed, abstract form as
\begin{align}\label{eq:rho_and_v_condensed}
\begin{split}
\partial_t v &= A_c(v,\rho)v + C_h(v,\rho)\rho + F_c(v,\rho)+Q(v,\rho) + G(v,\rho) F_h(v,\rho), \\
\partial_t \rho &= A_h(v,\rho)\rho + F_h(v,\rho),
\end{split}
\end{align}
where $C_h: X_{\gamma,\mu} \times V_\mu \to \mathcal{L}(X_1, X_0)$ is defined by
$$
C_h(v,\rho)\hat{\rho} := B_h(v,\rho)\hat{\rho} + G(v,\rho)A_h(v,\rho)\hat{\rho}.
$$

\section{Local well-posedness}\label{Local well-posedness}

We assume that $p, q \in (1, \infty)$ and $\mu \in (1/p, 1]$ satisfy \eqref{eq:Assump_pq}, so that the embedding
$$
X_{\gamma,\mu} = B_{qp}^{4\mu - 4/p}(\Sigma) \hookrightarrow C^2(\Sigma)
$$
is available to us. Furthermore, we assume that $\beta\in (3/4+n/(8q),1)$.
Recall the definition
$$
V_\mu = \{ \rho \in X_{\gamma,\mu} \mid \| \rho \|_{L_\infty(\Sigma)} < r \}.
$$
Now, we define the mappings $A: X_{\gamma,\mu} \times V_\mu \to \mathcal{L}(X_1 \times X_1, X_0 \times X_0)$, $F_1: X_{\gamma,\mu} \times V_\mu \to X_0 \times X_0$ and $F_2:X_\beta\times (V_\mu\cap X_\beta)\to X_0$ by
$$
A(v, \rho) :=
\begin{pmatrix}
A_c(v, \rho) & C_h(v, \rho) \\
0 & A_h(v, \rho)
\end{pmatrix},
$$
$$
F_1(v, \rho) :=
\begin{pmatrix}
F_c(v, \rho) + G(v, \rho) F_h(v, \rho) \\
F_h(v, \rho)
\end{pmatrix}
$$
and
$$
F_2(v, \rho) :=
\begin{pmatrix}
Q(v, \rho) \\
0
\end{pmatrix}.
$$

With these definitions, and setting $z = (v, \rho)$, we can reformulate \eqref{eq:rho_and_v_condensed} as the quasilinear parabolic evolution equation
\begin{equation}\label{eq:quasilin}
\partial_t z = A(z) z + F_1(z)+F_2(z),
\end{equation}
subject to the initial value $z(0) = z_0 = (v_0, \rho_0) \in X_{\gamma,\mu} \times V_\mu$. We claim that for any $v_0 \in X_{\gamma,\mu}$, the operator $A(z_0) = A(v_0, 0)$ possesses the property of $L_{p,\mu}$-maximal regularity in $X_0\times X_0$.

\begin{proposition}
Let $1 < p < \infty$, $\mu \in (1/p, 1]$ satisfy \eqref{eq:Assump_pq}, $0 < T < \infty$, and assume conditions (R) \& (P). Then, for any $v_0 \in X_{\gamma,\mu}$, the operator $A(v_0, 0)$ has maximal $L_{p,\mu}$-regularity in $X_0 \times X_0$. More precisely, for any $f \in L_{p,\mu}((0,T); X_0 \times X_0)$ there exists a unique solution
$$
z \in H_{p,\mu}^1((0,T); X_0 \times X_0) \cap L_{p,\mu}((0,T); X_1 \times X_1)
$$
to the linear problem
\begin{equation}\label{eq:MRlin1}
\partial_t z = A(v_0, 0) z + f, \quad t \in (0, T), \qquad z(0) = 0.
\end{equation}
\end{proposition}

\begin{proof}
Due to the triangular structure of $A(v, \rho)$, equation \eqref{eq:MRlin1} can be rewritten as
\begin{equation}\label{eq:MRlin2}
\partial_t v = A_c(v_0, 0) v + C_h(v_0, 0) \rho + f_c, \quad t \in (0, T), \qquad v(0) = 0,
\end{equation}
and
\begin{equation}\label{eq:MRlin3}
\partial_t \rho = A_h(v_0, 0) \rho + f_h, \quad t \in (0,T), \qquad \rho(0) = 0,
\end{equation}
where $f = (f_c, f_h)$. By Proposition \ref{pro:quasi_rho}, the operator $A_h(v_0, 0)$ is locally given by
$$
\varphi_* A_h(v_0, 0) \hat{\rho} = \sum_{3 \leq |\sigma| \leq 4} a_\sigma^h(\varphi_* v_0, \partial(\varphi_* v_0)) \partial^\sigma \hat{\rho}_*,
$$
with leading order term
$$
\sum_{|\sigma| = 4} a_\sigma^h(\varphi_* v_0) \partial^\sigma \hat{\rho}_* = -\frac{1}{\beta(0)} \alpha(\varphi_* v_0)
\sum_{i,j,k,\ell} a^{k\ell} c^{ij} \partial_i \partial_j \partial_k \partial_\ell \hat{\rho}_*,
$$
where $a^{k\ell} = a^{k\ell}(0, 0)$ and $c^{ij} = c^{ij}(0, 0)$. The local principal symbols of $A_h(v_0, 0)$ are given by
$$
-\frac{1}{\beta(0)} \alpha(\varphi_* v_0(x))
\sum_{i,j,k,\ell} a^{k\ell} c^{ij} \xi_i \xi_j \xi_k \xi_\ell,
$$
where $(x, \xi) \in \mathbb{R}^n \times \mathbb{R}^n$ with $\|x\|_2 \le 1$ and $\|\xi\|_2 = 1$. Since $\beta(0) = 1$, $\alpha > 0$ and the matrices $(a^{k\ell})$ and $(c^{ij})$ are symmetric and positive definite, it follows that $-A_h(v_0, 0)$ is parameter-elliptic of angle $0$ in the sense of \cite[Definition 6.4.1]{PruSim16}. Hence, $-A_h(v_0, 0)$ is normally elliptic and thus, by \cite[Theorem 6.4.3(ii)]{PruSim16}, there exists a unique solution
$$
\tilde{\rho} \in H_{p,\mu}^1((0,T); X_0) \cap L_{p,\mu}((0,T); X_1)
$$
for \eqref{eq:MRlin3}.

Substituting $\tilde{\rho}$ into \eqref{eq:MRlin2} results in the equation
\begin{equation}\label{eq:MRlin4}
\partial_t v = A_c(v_0, 0) v + \tilde{f}_c, \quad t \in (0, T), \qquad v(0) = 0,
\end{equation}
where $\tilde{f}_c := C_h(v_0, 0) \tilde{\rho} + f_c \in L_{p,\mu}((0,T); X_0)$ due to $v_0 \in X_{\gamma, \mu} \hookrightarrow C^2(\Sigma)$ and \eqref{eq:Assump_pq}. By Proposition \ref{pro:quasi_v}, in local coordinates, $A_c(v_0, 0)$ acts as
$$
\varphi_* A_c(v_0, 0) \hat{v} = \sum_{3 \leq |\sigma| \leq 4} a_\sigma^c(\varphi_* v_0, \partial(\varphi_* v_0)) \partial^\sigma \hat{v}_*,
$$
with leading order term
$$
\sum_{|\sigma| = 4} a_\sigma^c(\varphi_* v_0)\partial^\sigma \hat{v}_* = -\gamma m(\varphi_* v_0) \sum_{i, j, k, \ell} a^{k\ell} a^{ij} \partial_i \partial_j \partial_k \partial_\ell \hat{v}_*,
$$
where $a^{k\ell} = a^{k\ell}(0, 0)$. Thus, the local principal symbols are
$$
-\gamma m(\varphi_* v_0(x)) \sum_{i,j,k,\ell} a^{k\ell} a^{ij} \xi_i \xi_j \xi_k \xi_\ell,
$$
where $(x, \xi) \in \mathbb{R}^n \times \mathbb{R}^n$ as above. Since $m > 0$ and $(a^{k\ell})$ is symmetric positive definite, $-A_c(v_0, 0)$ is parameter-elliptic of angle $0$ in the sense of \cite[Definition 6.4.1]{PruSim16}. Therefore, \cite[Theorem 6.4.3(ii)]{PruSim16} ensures the existence of a unique solution
$$
\tilde{v} \in H_{p,\mu}^1((0,T); X_0) \cap L_{p,\mu}((0,T); X_1)
$$
for \eqref{eq:MRlin4}. Collecting our results, we conclude that the pair
$$
\tilde{z} := (\tilde{v}, \tilde{\rho}) \in H_{p,\mu}^1((0,T); X_0 \times X_0) \cap L_{p,\mu}((0,T); X_1 \times X_1)
$$
is the unique solution of \eqref{eq:MRlin1}.
\end{proof}

Since both $A$ and $F_1$ depend smoothly on derivatives of $(v, \rho)$ up to second order, and since $X_{\gamma,\mu} \hookrightarrow C^2(\Sigma)$ by \eqref{eq:Assump_pq}, we readily obtain
$$
A \in C^{1-}(X_{\gamma,\mu} \times V_\mu; \mathcal{L}(X_1 \times X_1, X_0 \times X_0)),
$$
as well as
$$
F_1 \in C^{1-}(X_{\gamma,\mu} \times V_\mu; X_0 \times X_0),
$$
since $\Sigma$ is compact. Concerning the nonlinearity $F_2(v,\rho)=(Q(v,\rho),0)^{\sf T}$, we show that there exist numbers $\beta\in (\mu-1/p,1)$, $\alpha_j\ge0$ and $\beta_j\in [\mu-1/p,\beta]$ with
\begin{equation}\label{LWP:AssF_20}
\frac{\alpha_j(\beta-\mu+1/p)+\beta_j-\mu+1/p}{1-\mu+1/p}<1,
\end{equation}
for $j\in\{1,2\}$, so that for each $\hat{v}\in X_{\gamma,\mu}$ and $R>0$ with $\bar{B}_R^{X_{\gamma,\mu}^2}(\hat{v},0)\subset X_{\gamma,\mu}\times V_\mu$ there exists $C_R>0$ such that the Gagliardo--Nirenberg-type estimate
\begin{multline}\label{LWP:AssF_2}
\|F_2(v_1,\rho_1)-F_2(v_2,\rho_2)\|_{X_0^2}\\
\le C_R\sum_{j=1}^2(1+\|(v_1,\rho_1)\|_{X_{\beta}^2}^{\alpha_j}+\|(v_2,\rho_2)\|_{X_{\beta}^2}^{\alpha_j})
\|(v_1-v_2,\rho_1-\rho_2)\|_{X_{\beta_j}^2},
\end{multline}
holds for all $(v_1,\rho_1),(v_2,\rho_2)\in \bar{B}_R^{X_{\gamma,\mu}^2}(\hat{v},0)\cap X_\beta^2$, where $X_{\beta_j}=(X_0,X_1)_{\beta_j,p}$ and $Y^2:=Y\times Y$ for $Y\in \{X_0,X_\beta,X_{\gamma,\mu}\}$.

Indeed, we recall from Proposition \ref{pro:quasi_v} the local representation
$$\varphi_*Q(v,\rho)=\sum_{|\sigma|=3}r_{\sigma}^h(v_*,\rho_*,\partial\rho_*)(\partial^\sigma\rho_*)^2,$$
with a smooth function $r_\sigma^h$. For the sake of readability, we drop the index $*$ in the following estimate. For all $(v_1,\rho_1),(v_2,\rho_2)\in \bar{B}_R^{X_{\gamma,\mu}^2}(\hat{v},0)\cap X_\beta^2$ and for each local chart $(U,\varphi)$, it holds that
\begin{align*}
\|r_{\sigma}^h(v_1,&\rho_1,\partial\rho_1)(\partial^\sigma\rho_1)^2
-r_{\sigma}^h(v_2,\rho_2,\partial\rho_2)(\partial^\sigma\rho_2)^2\|_{L_q(\varphi(U))}\\
&\le \sum_{|\sigma|=3}\|r_{\sigma}^h(v_1,\rho_1,\partial\rho_1)\|_{L_\infty(\varphi(U))}\|(\partial^\sigma\rho_1)^2
-(\partial^\sigma\rho_2)^2\|_{L_q(\varphi(U))}\\
&\quad + \sum_{|\sigma|=3}\|(\partial^\sigma\rho_2)^2\|_{L_q(\varphi(U))}
\|r_{\sigma}^h(v_1,\rho_1,\partial\rho_1)-r_{\sigma}^h(v_2,\rho_2,\partial\rho_2)\|_{L_\infty(\varphi(U))}\\
&\le C_R(\|\rho_1\|_{H_{2q}^3(\Sigma)}+\|\rho_2\|_{H_{2q}^3(\Sigma)})\|\rho_1-\rho_2\|_{H_{2q}^3(\Sigma)}\\
&\quad + C_R\|\rho_2\|_{H_{2q}^3(\Sigma)}^2\left(\|(v_1-v_2\|_{X_{\gamma,\mu}}+\|\rho_1-\rho_2)\|_{X_{\gamma,\mu}}\right),
\end{align*}
since $X_{\gamma,\mu}\hookrightarrow C^2(\Sigma)$ by \eqref{eq:Assump_pq}. With the help of a finite partition of unity ($\Sigma$ is compact), it then follows that
\begin{align*}
\|Q(v_1,\rho_1)-Q(v_2,\rho_2)\|_{L_q(\Sigma)}&\le C_R(\|\rho_1\|_{H_{2q}^3(\Sigma)}+\|\rho_2\|_{H_{2q}^3(\Sigma)})\|\rho_1-\rho_2\|_{H_{2q}^3(\Sigma)}\\
&\quad + C_R\|\rho_2\|_{H_{2q}^3(\Sigma)}^2\left(\|(v_1-v_2\|_{X_{\gamma,\mu}}+\|\rho_1-\rho_2)\|_{X_{\gamma,\mu}}\right),
\end{align*}
with some modified constant $C_R>0$. As in the proof of Proposition \ref{pro:quasi_v}, for $\beta\in (3/4+n/(8q),1)$, we have the embedding
$$X_\beta=B_{qp}^{4\beta}(\Sigma)\hookrightarrow H_{2q}^3(\Sigma)$$
at our disposal. This yields the estimate \eqref{LWP:AssF_2} for the choice $(\alpha_1,\beta_1)=(1,\beta)$ and $(\alpha_2,\beta_2)=(2,\mu-1/p)$, since $X_{\beta_j}\hookrightarrow X_{\gamma,\mu}$, by abstract interpolation theory.
For $j\in\{1,2\}$, the condition \eqref{LWP:AssF_20} then reads
\begin{equation}\label{LWP:beta}
\beta<\frac{1}{2}(1+\mu-1/p).
\end{equation}
From now on, let us choose
\begin{equation}\label{LWP:beta2}
\beta\in (3/4+n/(8q),1)\cap (\mu-1/p,1).
\end{equation}
We note that this intersection is not empty and moreover it holds that
$$\max\left\{\frac{3}{4}+\frac{n}{8q},\mu-1/p\right\}<\frac{1}{2}(1+\mu-1/p).$$
by assumption \eqref{eq:Assump_pq}. Thus, it is possible to find a number $\beta$ satisfying the conditions \eqref{LWP:beta} and \eqref{LWP:beta2}.

Therefore, we are in a position to apply \cite[Theorem 2.1]{LPW14} to equation \eqref{eq:quasilin}, yielding the following result.

\begin{theorem}\label{thm:LWP}
Let $p, q \in (1, \infty)$ and $\mu \in (1/p, 1]$ satisfy \eqref{eq:Assump_pq}, and assume (R) and (P). Suppose $v_0 \in B_{qp}^{4\mu - 4/p}(\Sigma)$. Then there exist $T = T(v_0) > 0$ and $\varepsilon = \varepsilon(v_0) > 0$ such that \eqref{CH-surface_transf} and \eqref{surfaceDiff_transf} admit a unique solution
$$
v, \rho \in H_{p,\mu}^1((0,T); L_q(\Sigma)) \cap L_{p,\mu}((0,T); H_q^4(\Sigma))\cap C([0,T];B_{qp}^{4\mu - 4/p}(\Sigma))
$$
for any $\rho_0 \in B_{qp}^{4\mu - 4/p}(\Sigma)$ with $\|\rho_0\|_{B_{qp}^{4\mu - 4/p}(\Sigma)}<\varepsilon$. Furthermore, it holds that $\|\rho(t)\|_{L_\infty(\Sigma)}<r$ for all $t\in [0,T]$, where $r>0$ is from Section \ref{sec:GS}.
\end{theorem}
\begin{remark}
Under the stronger assumption
\begin{equation}\label{eq:Assump_pq_2}
\frac{4}{p}+\frac{n}{q}<1\quad\text{and}\quad 3+\frac{4}{p}+\frac{n}{q}<4\mu
\end{equation}
on $p,q\in (1,\infty)$ and $\mu\in (1/p,1]$, compared to \eqref{eq:Assump_pq}, it holds that
$$X_{\gamma,\mu}=B_{qp}^{4\mu-4/p}(\Sigma)\hookrightarrow C^3(\Sigma),$$
hence
$$
A \in C^{1-}(X_{\gamma,\mu} \times V_\mu; \mathcal{L}(X_1 \times X_1, X_0 \times X_0)),
$$
as well as
$$
F_1,F_2 \in C^{1-}(X_{\gamma,\mu} \times V_\mu; X_0 \times X_0).
$$
Therefore, one may apply \cite[Theorem 2.1]{KPW10} to \eqref{eq:quasilin}, without justification of a Gagliardo-Nirenberg-type estimate for $F_2$, leading to the same assertion as in Theorem \ref{thm:LWP} but under the stronger condition \eqref{eq:Assump_pq_2}.
\end{remark}

We conclude this section with a short discussion of the trace spaces $B_{qp}^{4\mu - 4/p}(\Sigma)$:

\begin{remark}
Let $\varepsilon := 4\mu - (2 + 4/p + n/q)$. Then
$$
B_{qp}^{4\mu - 4/p}(\Sigma) = B_{qp}^{2 + n/q + \varepsilon}(\Sigma)
$$
and, by \eqref{eq:Assump_pq}, $\varepsilon > 0$. Note that $\varepsilon$ can be chosen arbitrarily small by taking $4\mu$ sufficiently close to $2 + 4/p + n/q$. In the most relevant case $n=2$ and $q=2$, this becomes
$$
B_{2p}^{4\mu - 4/p}(\Sigma) = B_{2p}^{3+\varepsilon}(\Sigma)
$$
where $p > 4$ so that $4/p + n/q < 2$ in \eqref{eq:Assump_pq}. Observe also that
$$
H_2^{3+\varepsilon}(\Sigma) = B_{22}^{3+\varepsilon}(\Sigma) \hookrightarrow B_{2p}^{3+\varepsilon}(\Sigma),
$$
where $H_r^s$ denotes a Bessel potential space.

Moreover, for any fixed $\delta > \varepsilon$, there exists a (large) $q > 1$ so that
$$
B_{qp}^{4\mu - 4/p}(\Sigma) = B_{qp}^{2+\delta}(\Sigma).
$$
Indeed, setting $\delta = n/q + \varepsilon$ gives $q = n / (\delta - \varepsilon)$. Here the requirement $4/p + n/q < 2$ becomes $4/p < 2 - (\delta - \varepsilon)$, so $p > 2$, since $\delta - \varepsilon > 0$ can be made arbitrarily small. Finally, if $s > 2+\delta$ and $s - n/p \ge 2+\delta - n/q$, then
$$
H_p^{s}(\Sigma) \hookrightarrow B_{qp}^{2+\delta}(\Sigma).
$$
In the special case $p = q$, this yields
$$
H_p^s(\Sigma) \hookrightarrow B_{pp}^{2+\delta}(\Sigma) = W_p^{2+\delta}(\Sigma),
$$
provided $s > 2+\delta$. We note that $s>2+\delta$ may be chosen arbitrarily close to 2, by choosing $\delta>\varepsilon>0$ sufficiently close to 0.

All of the above mentioned properties and embeddings follow e.g. from \cite{Tri78} in combination with \cite[Theorem 4.4]{Ama25}.
\end{remark}

\end{document}